\documentclass[11pt]{amsart}
\usepackage{amsfonts,amssymb,amsmath,amsthm,amscd}
\usepackage[textwidth=17cm,textheight=20cm]{geometry}

\newcommand{\beq}{\begin{equation}}
\newcommand{\eeq}{\end{equation}}
\newcommand{\bea}{\begin{eqnarray}}
\newcommand{\eea}{\end{eqnarray}}
\newcommand{\noi}{\noindent}
\newcommand{\nn}{\nonumber}

\usepackage{longtable}
\usepackage{graphicx}
\usepackage{amsmath}
\newtheorem{pro}{Proposition}
\newtheorem{lemma}{Lemma}

\newtheorem{definition}{Definition}
\newtheorem{theorem}{Theorem}

\newcommand{\C}{\mathbb{C}}

\begin{document}
\title{Singularity confinement for matrix discrete Painlev\'{e} equations}
\author{Giovanni A. Cassatella-Contra} \address{Departamento de F\'{\i}sica
  Te\'{o}rica II (M\'{e}todos Matem\'{a}ticos de la F\'{\i}sica), Facultad de
  F\'{\i}sicas, Universidad Complutense de Madrid, 28040 Madrid, Spain}
  \thanks{GC-C benefitted of the financial support of a ``Acci\'{o}n Especial'' Ref. AE1/13-18837 of the Universidad Complutense de Madrid}
\email{gaccontra@fis.ucm.es} \author{ Manuel Ma\~{n}as}
\thanks{MM thanks economical support from the Spanish ``Ministerio de Econom\'{\i}a y Competitividad" research project MTM2012-36732-C03-01,  \emph{Ortogonalidad y aproximacion; Teoria y Aplicaciones}}\email{manuel.manas@ucm.es}\author{Piergiulio Tempesta}\address{Departamento de F\'{\i}sica
  Te\'{o}rica II (M\'{e}todos Matem\'{a}ticos de la F\'{\i}sica), Facultad de
  F\'{\i}sicas, Universidad Complutense de Madrid, 28040 Madrid, Spain and Instituto de Ciencias Matem\'{a}ticas, C/ Nicol\'{a}s Cabrera, No 13--15, 28049 Madrid, Spain}\thanks{ PT has been
supported by  Spanish ``Ministerio de Ciencia e Innovaci\'{o}n" grant FIS2011--00260.}
\email{p.tempesta@fis.ucm.es}
%\date{\today}
\maketitle
\begin{abstract}
We study the analytic properties of a matrix discrete system introduced in \cite{CM}. The singularity confinement for this system
is shown to hold generically, i.e. in the whole space of parameters except possibly for algebraic subvarieties. This paves the way to a generalization of Painlev\'{e} analysis to discrete matrix models.
\end{abstract}
%\tableofcontents
\section{Introduction}
Since the discovery of the \textit{Painlev\'e property} for ordinary differential equations at the end of
the XIX century \cite{Pain}, the notion of
 \textit{integrability} has been
related to the local analysis of movable isolated singularities of solutions
of dynamical systems
 \cite{Conte}. This approach to integrability has opened
an alternative perspective compared to the standard algebraic approach
\textit{\`a la
 Liouville}, based on the existence of a suitable number of
functionally independent integrals of motion. Both points of view have been
extended to the study of evolution equations on a discrete background.

Integrable discrete systems, for several aspects more fundamental objects than
the continuous ones, are ubiquitous both in pure
 and applied mathematics,
and in theoretical physics as well. They possess rich algebraic--geometric
properties \cite{AB}, \cite{BS1}, \cite{MV}, \cite{Nov3}, \cite{Tsuda} and are
relevant, for instance, in the regularization of quantum field theories in a
lattice and in discrete quantum
 gravity \cite{FIK}, \cite{thooft}. %Frobenius
%manifolds are useful in the discussion of generalized Toda systems
%\cite{CDZ}.

 In particular, the problem of integrability
 preserving
discretizations of partial differential equations has become a very active research area \cite{Suris},
and has been widely investigated with both geometrical and algebraic methods
\cite{BS1}, \cite{BS2}, \cite{Nov1}, \cite{Temp}.

%Physically relevant discrete integrable models have been considered in the context of field theories and Hamiltonian gravity, for instance in \cite{GP}, \cite{thooft}, \cite{RS}.

The approach known as \textit{singularity confinement}, introduced in \cite{GRP}, is the equivalent for discrete systems of the singularity analysis for continuous dynamical systems. It essentially relies on the observation that for integrable discrete models, if a singularity appears in some specific point of the lattice of the independent variable, then it would disappear after making evolve the system via a finite number of iterations. Alternative, related approaches are based on the notion of algebraic entropy \cite{BV}, \cite{LRG} or on Nevalinna theory \cite{AHH}, \cite{RGTT}. A large class of difference equations coming from unitary integrals and combinatorics possess the confinement property \cite{AMV}.
However, observe that singularity confinement, in despite of being extremely useful in isolating integrability, it might not be a sufficient condition for integrability as was noticed by Hietarinta and Viallet \cite{hietarinta-viallet}.

%singularity analysis has become a standard tool in the study of the integrability of dynamical systems \cite{Conte}.

The purpose of this paper is to start a theoretical study of the singularity confinement property for \textit{matrix integrable systems}.
Indeed, we hypothesize that the singularity
analysis has for matrix systems the same relevance that possesses for both discrete and continuous scalar models.

Apart its intrinsic mathematical interest, the study of matrix discrete dynamical systems can also be related, from an applicative point of view, to the theory of complex networks \cite{Newbook}. Indeed, given a
random graph with $N$
vertices, one associates with it the adjacency matrix, which is a $N\times N$ matrix, whose entries $a_{ij}$ represent the number of links associated with the nodes $i$ and $j$ $(i,j=1,\ldots, N)$. The discrete time evolution of the topology of the network would provide a difference equation for the adjacency matrix, defining a discrete matrix model.

Hereafter, we shall focus on the singularity confinement of the following discrete matrix equation
\begin{eqnarray}
 \beta_{n+1}&=&n\beta_{n}^{-1}-\beta_{n-1}-\beta_{n}-\alpha,  \qquad n=1,2,\ldots
%\noindent\beta_{0}&=&0,
\label{eq:painleveS}
\end{eqnarray}
where $\beta_n\in\C^{N\times N}$ is a $N \times N$ complex matrix.

Equation \eqref{eq:painleveS} can be considered a kind of non Abelian matrix version of the discrete Painlev\'e equation
(dPI). It has been introduced in \cite{CM} and describes the recursion relation for the matrix coefficients of a
class of Freud matrix orthogonal polynomials with a quartic potential \cite{freud}. It is obtained by solving the related Riemann--Hilbert
problem.
In that paper we also proved the singularity confinement in a simple situation, when the initial data are triangular matrices up to similarity transformations. The aim of this paper is to extend this result to the general case.
 This  extension have required much more effort that in the simple triangularizable situation but finally we succeeded in getting the desired proof. The difficulty  mainly resides in the analysis of the genericness of the result given in Theorem \ref{generico}.

\subsection{Preliminary discussion}

Let us present here the simplest case of singularity analysis for the matrix model \eqref{eq:painleveS} which parallels the results for the standard discrete Painlev\'{e} I equation.
We assume that
\begin{align}
\beta_{m-1}&=\beta_{m-1,0}+
\beta_{m-1,1}\epsilon+O(\epsilon^2), &\beta_{m}&=\beta_{m,1}\epsilon+\beta_{m,2}\epsilon^{2}
+O(\epsilon^3),&\epsilon&\rightarrow0,
\label{eq:betaunoII}
\end{align}
with $\det\beta_{m,1}\neq0$.
If we introduce conditions
\eqref{eq:betaunoII} into \eqref{eq:painleveS}, we have that
\begin{align}
\beta_{m+1}=m\beta_{m,1}^{-1}\epsilon^{-1}
+\beta_{m+1,0}+\beta_{m+1,1}\epsilon+\beta_{m+1,2}\epsilon^2+O(\epsilon^3),
\label{eq:C2primow}
\end{align}
 where
\begin{align*}
\beta_{m+1,0}=&-m\beta_{m,1}^{-1}\beta_{m,2}\beta_{m,1}^{-1}
-\beta_{m-1,0}-\alpha,\\
\beta_{m+1,1}=&m\beta_{m,1}^{-1}(\beta_{m,2}\beta_{m,1}^{-1}\beta_{m,2}
-\beta_{m,3})\beta_{m,1}^{-1}-\beta_{m,1}-\beta_{m-1,1},\\
\beta_{m+1,2} =&m\Big(\beta_{m,2}\beta_{m,1}^{-1}(\beta_{m,3}
-\beta_{m,2}\beta_{m,1}^{-1}\beta_{m,2})+\beta_{m,3}\beta_{m,1}^{-1}\beta_{m,2}-\beta_{m,4}\Big)\beta_{m,1}^{-2}-\beta_{m,2}-\beta_{m-1,2}.
\end{align*}
 Thus, the pole singularity has shown up, and it will survive still for another step in the sequence. Indeed, we have that
\begin{align}
\beta_{m+2}=-m\beta_{m,1}^{-1}\epsilon^{-1}
+\beta_{m+2,0}+\beta_{m+2,1}\epsilon+\beta_{m+2,2}\epsilon^2+O(\epsilon^3),
\label{eq:C3primow}
\end{align}
 where
\begin{align*}
\beta_{m+2,0}=&m\beta_{m,1}^{-1}\beta_{m,2}\beta_{m,1}^{-1}+\beta_{m-1,0},\\
\beta_{m+2,1}=&\frac{(m+1)}{m}\beta_{m,1}-
m\beta_{m,1}^{-1}\beta_{m,2}\beta_{m,1}^{-1}\beta_{m,2}\beta_{m,1}^{-1}
+m\beta_{m,1}^{-1}\beta_{m,3}\beta_{m,1}^{-1}+\beta_{m-1,1},\\
\beta_{m+2,2}=&\frac{(m+1)}{m}\beta_{m,2}
+\frac{(m+1)}{m^2}\beta_{m,1}(\beta_{m-1,0}+\alpha)\beta_{m,1}+m\beta_{m,2}\beta_{m,1}^{-1}(\beta_{m,2}\beta_{m,1}^{-1}
\beta_{m,2}\beta_{m,1}^{-2}-\beta_{m,3}\beta_{m,1}^{-2})\\
&-m\beta_{m,3}\beta_{m,1}^{-1}\beta_{m,2}\beta_{m,1}^{-2}
+\beta_{m-1,2}+m\beta_{m,4}\beta_{m,1}^{-2}.
\end{align*}

\noi We easily check that in the third step the zero appears again
\begin{align}
\beta_{m+3}=&\frac{-(m+3)}{m}\beta_{m,1}\epsilon+\beta_{m+3,2}\epsilon^{2}
+O(\epsilon^3),
\label{eq:C4primow}
\end{align}
 where
\begin{align*}
\beta_{m+3,2}:=&-\frac{(m+3)}{m}\beta_{m,2}
-\frac{(2m+3)}{m^2}\beta_{m,1}\beta_{m-1,0}\beta_{m,1}
-\frac{(m+1)}{m^2}\beta_{m,1}\alpha\beta_{m,1}.
\end{align*}
 Finally, if we substitute \eqref{eq:C3primow} and \eqref{eq:C4primow} into
 \eqref{eq:painleveS} we get no singularity at all:
\begin{align*}
\beta_{m+4}=\frac{m}{(m+3)}\beta_{m-1,0}-\frac{2}{(m+3)}\alpha
+O(\epsilon).
%\label{eq:C5primow}
\end{align*}
 Observe that $\beta_{m+3}=O(\epsilon)$, $\beta_{m+4}=O(1)$ and
 $\det\beta_{m+4}=O(1)$ for $\epsilon\rightarrow0$. Thus, unless
\begin{align}
  \label{non}
\det(m\beta_{m-1,0}-2\alpha)=0,
\end{align}
we obtain singularities in the step just after the appearance of a \emph{zero}
in $\beta_{m}$, with the poles appearing in the sites $m+1,m+2$. Then we have
a \emph{zero} for $m+3$ while we recover the standard behaviour for $m+4$. A
crucial point is that this singularity confinement holds whenever \eqref{non}
is \textit{not} satisfied. This observation motivates the definitions proposed
in the following discussion.
\begin{definition}
Whenever the singularity confinement property is satisfied in the whole space
$\mathcal{S}$ of
 parameters except possibly for a set of algebraic
subvarieties $\mathcal{W}_i\in\mathcal{S}, i=1,,,j\in\mathbb{N}$, we shall
say that the property is satisfied \textit{generically}.
\end{definition}
 In this case we will speak about the genericness of the singularity
 confinement.

\begin{definition}
 We shall define the confinement time as the minimum number $\ell\in\mathbb N$
 of iterations or steps in the lattice, after the appearance of
 a zero,
 necessary to recover the form without poles or zeros.
\end{definition}

Thus, in the above case we have generically a singularity confinement with a confinement time $l=4$.

A simple but fundamental observation for the sequel of the paper is the following one.
\begin{lemma}
The matrix system (\ref{eq:painleveS}) is invariant under similarity transformations.
\begin{proof}
Observe that
\[
M\beta_{n+1}M^{-1}=nM\beta_{n}^{-1}M^{-1}-M\beta_{n-1}M^{-1}-M\beta_{n}M^{-1}
-M{\alpha}M^{-1}.
\]
Therefore, we obtain
\begin{equation}
\phi_{n+1}=n\phi_{n}^{-1}-\phi_{n-1}-\phi_{n}-\delta,
%\label{eq:semejantepainleveS}
\nonumber
\end{equation}
\noindent where $\phi_{n}${$:=$}$M\beta_{n}M^{-1}$ and
$\delta${$:=$}$M\alpha{M}^{-1}$.
\end{proof}
\end{lemma}

\subsection{Main result}
The ideas developed within this example will be used in the subsequent
considerations to study the confinement of the singularities of the matrix dPI
model \eqref{eq:painleveS}.
 We shall assume that
\begin{align}
\beta_{m-1}&=\beta_{m-1,0}+
\beta_{m-1,1}\epsilon+O(\epsilon^2),& &&\epsilon&\rightarrow0,
\label{eq:betacero}
\\
\beta_{m}&=\beta_{m,0}+\beta_{m,1}\epsilon+O(\epsilon^2),&  \det\beta_{m}&=O(\epsilon^r),&
\epsilon&\rightarrow0,
\label{eq:betauno}
\end{align}
where $\beta_{m-1,i},\beta_{m,i}\in\C^{N\times N}$ and
$r=1,\dots,N$. Consequently, we can distinguish two cases.
\begin{itemize}
\item{$r=N$.} This is the maximal rank case; for it we have that
\begin{align*}
\beta_{m,0}&=0,&\det\beta_{m,1}&\neq0.
%\label{eq:condicionA}
\end{align*}
It presents singularity confinement generically.
\item{$r\leq{N-1}$.} For the non-maximal rank case we instead have
\begin{align}
\dim\operatorname{Ran}\beta_{m,0}&=N-r,\nonumber\\
  \det\beta_{m}&=O(\epsilon^r),&\epsilon&\rightarrow0.
\label{eq:condicionB}
\end{align}
\end{itemize}
As will be proven later, by using the invariance under a similarity
transformation, one can assume that the matrices $\beta$ will have the form
expressed by eq. \eqref{eq:betauno2}. So said, we can state the main result of
the paper as follows.
\begin{theorem}\label{main}
 If $\beta_{m-1}$ and $\beta_{m}$ are of the form \eqref{eq:betacero},
 \eqref{eq:betauno} and \eqref{eq:betauno2}, and the following conditions for
 $\epsilon\to 0$ are satisfied
\begin{align}
\det\beta_{m+1}=O(\epsilon^{-r}),
\label{eq:invert1}\\
                     \det\beta_{m+2}=O(\epsilon^{-r}),
\label{eq:invert2}\\
\det\beta_{m+3}=O(\epsilon^{r}),
\label{eq:invert3}\\
\det\beta_{m+4}=O(1),
\label{eq:invert4}\end{align}
then, there is singularity confinement for the dPI model  \eqref{eq:painleveS} with
confinement time $l=4$.
\end{theorem}

\section{$N${$\times$}$N$ matrix asymptotic expansions and singularity confinement}\label{a2}

 \noi In this section we will consider the set of matrix asymptotic expansions
\begin{eqnarray*}
\mathcal A= \C^{N\times N}(\!(\epsilon)\!)&:=\big\{M_0+M_1\epsilon+O(\epsilon^2),\,\epsilon\to 0,
  M_i\in{\mathbb C}^{N\times N}\big\}.\label{eq:tipoT}
\end{eqnarray*}

 \noi This set is a ring with
 identity, given by the matrix ${\mathbb
  I}_N$. For each possible rank (see eq. \eqref{eq:betauno2}) $r=1,\dots, N-1$
we will use the block notation
\begin{align}
\nn
\label{Mblocks}
M&:=\begin{pmatrix}
  A & B   \\
 C  & D
\end{pmatrix},& A\in\C^{r\times r}, B\in\C^{r\times(N-r)},C\in\C^{(N-r)\times r}, D\in\C^{(N-r)\times(N-r)}.
\end{align}
We also introduce two subalgebras of the algebra  $\C^{N\times N}$
{\nn\begin{align}
  \mathfrak K&:=\Big\{ K=\begin{pmatrix}
  0 & 0   \\
 K_{21}  &  K_{22}
\end{pmatrix},
K_{21}\in{\mathbb C}^{(N-r)\times r}, K_{22}\in{\mathbb
  C}^{(N-r)\times(N-r)}\Big\},\\ \mathfrak L&:=\{ L=\begin{pmatrix}
L_{11} & 0 \\ L_{21} & 0
\end{pmatrix},\;   L_{11}\in{\mathbb C}^{r\times r},
L_{21}\in{\mathbb C}^{(N-r)\times {r}}\Big\},
 \end{align}}
and the related subsets of matrix asymptotic expansions
\begin{align*}
  \mathcal A_{\mathfrak K}&:=\{K\in\mathcal A, K|_{\epsilon=0}\in\mathfrak
  K\},
\label{eq:tipoK}&\mathcal A_{\mathfrak L}&:=\{L\in\mathcal A, L|_{\epsilon=0}\in\mathfrak L\},
%\label{eq:tipoL}
\end{align*}
which satisfy several important properties.
\begin{pro}\label{pro:rings}
The following  statements hold.
  \begin{enumerate}
    \item Both $ \mathcal A_{\mathfrak K}$ and $ \mathcal A_{\mathfrak L}$ are
      subrings without identity of the
 ring $\mathcal A$.
    \item For $K\in{\mathcal A}_{\mathfrak K}$ such that
      $\det{K}=O(\epsilon^{r})$, $\epsilon\rightarrow0$, then
 $ K^{-1}\in
      \epsilon^{-1}\mathcal A_{\mathfrak L}$, and reciprocally if
      $L\in\epsilon^{-1} \mathcal A_{\mathfrak L}$ with
      $\det{L}=O(\epsilon^{-r})$, $\epsilon\rightarrow0$, then
      $L^{-1}\in\mathcal
 A_{\mathfrak K}$.
      \item If $K \in\mathcal A_{\mathfrak K}$, that is
        $K=\Big(\begin{smallmatrix}
        0&0\\C_0&D_0\end{smallmatrix}\Big)+\Big(\begin{smallmatrix}
          A_1&B_1\\C_1&D_1\end{smallmatrix}\Big)\epsilon+O(\epsilon^2)$
 then
          {\nn\begin{align}
 \det K&=\epsilon^{r}\det\begin{pmatrix}
 A_{1}
            & B_{1}\\
 C_{0} & D_{0}
\end{pmatrix}+O(\epsilon^{r+1}), & \epsilon&\to 0.
%\label{eq:}
\end{align}}
\item If $L\in\epsilon^{-1}\mathcal A_{\mathfrak L}$, that is
  $L=\Big(\begin{smallmatrix}
  A_0&0\\C_0&0\end{smallmatrix}\Big)\epsilon^{-1}+\Big(\begin{smallmatrix}
    A_1&B_1\\C_1&D_1\end{smallmatrix}\Big)+O(\epsilon)$
 then
    {\nn\begin{align}
 \det L&=\epsilon^{-r}\det\begin{pmatrix}
 A_{0} &
      B_{1}\\
 C_{0} & D_{1}
\end{pmatrix}+O(\epsilon^{-r+1}), & \epsilon&\to 0.
%\label{eq:}
\end{align}}
    \item The subrings $\mathcal A_{\mathfrak K}$ and $\mathcal A_{\mathfrak
      L}$ are right and
 left ideals of $\mathcal A$ respectively, i.e.
      $\mathcal A_{\mathfrak K}\cdot\mathcal A\subset\mathcal
 A_{\mathfrak
      K}$ and $\mathcal A\cdot\mathcal A_{\mathfrak
 L}\subset\mathcal
      A_{\mathfrak L}$.
    \item The following inclusion holds: $\epsilon^{-1}\mathcal A_{\mathfrak L}\cdot\mathcal
      A_{\mathfrak K}\subset\mathcal A$.
  \end{enumerate}
\end{pro}
The proof of the previous statements is direct and left to the reader.

 To study the singularity confinement of the matrix equation \eqref{eq:painleveS} when
$\beta_{n}$ satisfies conditions \eqref{eq:condicionB}, we shall use expressions
\eqref{eq:betacero} and \eqref{eq:betauno}, having applied a similarity
transformation to $\beta$ such that $\beta_{m,0}\in\mathfrak K$, $\beta_m\in\mathcal A_{\mathfrak K}$. In other words
\begin{align}
\label{eq:betauno2}
\beta_{m,0}=\begin{pmatrix}
 0 & 0       & \cdots    & 0 & 0 &\cdots &0\\[8pt]
0 & 0 & \cdots & 0 & 0 &\cdots&0\\[8pt]
\vdots & \vdots & & \vdots & \vdots & &\vdots\\[8pt]
\beta_{m,0;r+1,1} & \beta_{m,0;r+1,2} & \cdots& \beta_{m,0;r+1,r+1}&
\beta_{m,0;r+1,r+2} & \cdots &  \beta_{m,0;r+1,N}\\[8pt]
\beta_{m,0;r+2,1} & \beta_{m,0;r+2,2} & \cdots & \beta_{m,0;r+2,r+1} &
\beta_{m,0;r+2,r+2} & \cdots &  \beta_{m,0;r+2,N}\\[8pt] \vdots & \vdots & &
\vdots&\vdots &  &\vdots \\[8pt] \beta_{m,0;N,1} &
\beta_{m,0;N,2} & \cdots& \beta_{m,0;N,r+1} & \beta_{m,0;N,r+2} & \cdots
&\beta_{m,0;N,N} \\[8pt]
\end{pmatrix},
\end{align}
where $m${$\geq$}2, and all the entries that are above the $r$+1-th
row of $\beta_{m}$ are zero. Notice that $\beta_{m-1}$ and $\beta_{m}$ belong to
the rings $\mathcal A$ and $\mathcal A_{\mathfrak K}$, respectively.

%\begin{theorem}\label{teorema}
% If $\beta_{m-1}$ and $\beta_{m}$ are of the form \eqref{eq:betacero},
%\eqref{eq:betauno} and \eqref{eq:betauno2}, and the following conditions for $\epsilon\to 0$ are satisfied
%\begin{align}
%\det\beta_{m+1}=O(\epsilon^{-r}),
%\label{eq:invert1}\\
%\det\beta_{m+2}=O(\epsilon^{-r}),
%\label{eq:invert2}\\
%\det\beta_{m+3}=O(\epsilon^{r}),
%\label{eq:invert3}\\
%\det\beta_{m+4}=O(1).
%\label{eq:invert4}\end{align}
%then there is singularity confinement for \eqref{eq:painleveS} with confinement time of $4$.
%\end{theorem}
\subsection{Proof of the Theorem \ref{main}}
\begin{proof}
As $\beta_{m,0}\in\mathfrak K$, i.e. $\beta_{m}\in\mathcal A_{\mathfrak
 K}$,
and by hypothesis $\det\beta_{m}=O(\epsilon^{r})$, $\epsilon\to 0$,
Proposition \ref{pro:rings}
 implies
\begin{align}
\beta_{m}^{-1}&=(\beta_m^{-1})_{-1} \epsilon^{-1}+(\beta_m^{-1})_{0}+O(\epsilon),& \epsilon&\to 0, &(\beta_m^{-1})_{-1}\in\mathfrak L.
\label{eq:betamenosuno}
\end{align}
 If we replace eqs. \eqref{eq:betauno} and
\eqref{eq:betauno2} into eq. \eqref{eq:painleveS}, and  take into account
eq. \eqref{eq:betacero}, we deduce
\begin{align*}
\beta_{m+1}&=m\beta_{m}^{-1}+O(1),&\epsilon&\rightarrow0.
%\label{eq:anillo}
\end{align*}
By using the relations \eqref{eq:betamenosuno}, \eqref{eq:betacero},
\eqref{eq:betauno} and \eqref{eq:betauno2}, this expression is reduced to
\begin{align}
\beta_{m+1}&=m(\beta_m^{-1})_{-1}\epsilon^{-1}+O(1),& \epsilon\to 0.
\label{eq:sing1}
\end{align}

 \noi Since $(\beta_m^{-1})_{-1}\in\mathfrak L$, from \eqref{eq:sing1} we
conclude
 that $\beta_{m+1}\in\epsilon^{-1}\mathcal A_{\mathfrak L}$, showing
a simple pole singularity. Due to the fact that by hypothesis
eq. \eqref{eq:invert1} holds,
 Proposition \ref{pro:rings} implies
\begin{align}
\beta_{m+1}^{-1}\in\mathcal A_{\mathfrak K}.
\label{eq:betamenosuno2}
\end{align}
 %(see Appendix \ref{a2}).
Then we deduce
\begin{align*}
\beta_{m+2}&=-m(\beta_m^{-1})_{-1}\epsilon^{-1}+O(1),&\epsilon&\rightarrow0,&\beta_{m+2}&\in\mathcal A_{\mathfrak L}.
%\label{eq:anillo2}
\end{align*}
As before, by using condition
\eqref{eq:invert2}, Proposition \ref{pro:rings} gives
\begin{eqnarray}
\nn
\beta_{m+2}^{-1}\in\mathcal A_{\mathfrak K}.
\label{eq:betamenosuno2-}
\end{eqnarray}
Now,
\begin{align}
\beta_{m+3}=\beta_{m}-(m+1)\beta_{m+1}^{-1}+(m+2)\beta_{m+2}^{-1},
\label{eq:painleves}
\end{align}
 where in the r.h.s. we have used twice eq. \eqref{eq:painleveS}
 to write
 $\beta_{m+2}$ as a function of $\beta_{m+1}$ and $\beta_{m}$.  As we have
 proven that $\beta_m,\beta_{m+1}^{-1},\beta_{m+2}^{-1}\in\mathcal A_{\mathfrak K}$,
we deduce that
\begin{align*}
\beta_{m+3}\in\mathcal A_{\mathfrak K}.
%\label{eq:sing3}
\end{align*}
As a consequence of eq. \eqref{eq:invert3} and Proposition \ref{pro:rings}, we obtain
\begin{align}
\beta_{m+3}^{-1}\in\epsilon^{-1}\mathcal A_{\mathfrak L}.
\label{eq:singularidad3}
\end{align}
Our matrix discrete Painlev\'{e} equation \eqref{eq:painleveS}
gives
\begin{align*}
\beta_{m+4}=(m+3)\beta_{m+3}^{-1}-\beta_{m+2}-\beta_{m+3}-\alpha,
%\label{eq:painleveS4}
\end{align*}
which implies
\begin{align}
\beta_{m+4}&=\beta_{m+3}^{-1}A+O(1),&\epsilon&\rightarrow0,& A&
:=(m+3){\mathbb I}_{N}-\beta_{m+3}\beta_{m+2},
\label{eq:confi}
\end{align}
where we have taken into account that $\beta_{m+3}$ and $\alpha$ are
$O(1)$. We study the matrix $A$, by applying eq. \eqref{eq:painleveS}
once. We get
\begin{align}
 A&={\mathbb I}_{N}+[(m+1)\beta_{m+1}^{-1}-\beta_{m}]\beta_{m+2}\nonumber\\
 &=[(m+1)\beta_{m+1}^{-1}
   -\beta_{m}][(m+1)\beta_{m+1}^{-1}-\beta_{m}-\alpha]-m{\mathbb
   I}_{N}+\beta_{m}\beta_{m+1} \nn
   \\&=[(m+1)\beta_{m+1}^{-1}
     -\beta_{m}][(m+1)\beta_{m+1}^{-1}-\beta_{m}-\alpha]-\beta_{m}(\beta_{m}
   +\beta_{m-1}+\alpha).
\label{eq:decisivo}
\end{align}
Now, recalling that $\beta_{m-1}=O(1)$, $\beta_m,
\beta_{m+1}^{-1}\in\mathcal A_{\mathfrak K}$, and by virtue of Proposition
\ref{pro:rings}
 we conclude that
\begin{align}\label{eq:Ak}
  A\in\mathcal A_{\mathfrak K}.
\end{align}
Finally, from eqs. \eqref{eq:singularidad3}, \eqref{eq:confi} and \eqref{eq:Ak} we
deduce that
\begin{align*}
  \beta_{m+4}\in\mathcal A.
\end{align*}
By taking into account that $\det\beta_{m+4}=O(1)$, we have proven that the
singularity has disappeared. Thus, the singularity confinement is ensured
with a confinement time $l=4$.
\end{proof}
%Now we proceed to discuss the issue of genericness of this result.

%\section{Genericness of the singularity confinement in the non-maximal rank $N\times N$ case}\label{genericness}

 In order to show the genericness of conditions
\eqref{eq:invert1}-\eqref{eq:invert4} we shall perform an asymptotic analysis, by introducing the expansions
\begin{align*}
\beta_{m-1}&=\sum_{i=0}^{\infty}\begin{pmatrix}
A_{m-1,i} &
B_{m-1,i} \\[8pt]
C_{m-1,i} &  D_{m-1,i}
\end{pmatrix}\epsilon^{i},
%\label{eq:Cblok0}
\\
\beta_{m}&=\begin{pmatrix}
0 & 0  \\[8pt]
C_{m,0} & D_{m,0}
\end{pmatrix}+
\sum_{i=1}^{\infty}\begin{pmatrix}
A_{m,i} & B_{m,i}  \\[8pt]
C_{m,i} &  D_{m,i}
\end{pmatrix}\epsilon^{i},
%\label{eq:Cblok1}
\end{align*}
 whereas $\alpha$ is written simply as
\begin{align*}
\alpha=\begin{pmatrix}
\alpha_{11} & \alpha_{12}  \\[8pt]
\alpha_{21} & \alpha_{22}
\end{pmatrix}.
%\label{eq:blokalfa}
\end{align*}
%Assuming $\det D_{m,0}\neq 0$ we define
%$S_{D}(\beta_m)_1:=A_{m,1}-B_{m,1}D_{m,0}^{-1}C_{m,0}\in\C^{r\times r}$.
%\subsection{On $\beta_{m+1},\beta_{m+2},\beta_{m+3}$ and $\beta_{m+4}$}
We introduce
\begin{definition}
  The following matrices will be useful
  \begin{align*}
    Z_{1}&:=D_{m+1,0}+D_{m,0}^{-1}C_{m,0}B_{m+1,0},\\
    Z_{2}&:=D_{m+2,0}+D_{m,0}^{-1}C_{m,0}B_{m+2,0},\\
    Z_{3}&:=D_{m+3,0}.
  \end{align*}
\end{definition}

The genericness  of the singularity confinement is described by

\begin{theorem}\label{generico}\begin{enumerate}
\item If $\det D_{m,0}\neq 0$, for $\epsilon\to 0$ we have
\begin{align*}
  \det\beta_{m+1}=O(\epsilon^{-r})\Leftrightarrow \det (Z_1)\neq 0.
\end{align*}
\item If ${\det}D_{m,0}\neq0$ , ${\det}Z_{1}\neq0$, we have that for
$\epsilon\to 0$
\begin{align*}
  \det\beta_{m+2}=O(\epsilon^{-r})\Leftrightarrow \det (Z_2)\neq 0.
  \end{align*}
  \item If $\det{D_{m,0}}{\neq}0$, ${\det}Z_{1}{\neq}0$ and
${\det}Z_{2}{\neq}0$, we have that for $\epsilon\to 0$
\begin{align*}
  \det\beta_{m+3}=O(\epsilon^{r})\Leftrightarrow {\det}Z_3\neq 0.
\end{align*}
\item If $\det D_{m,0}\neq 0$, $\det Z_1\neq 0$, $\det Z_2\neq 0$ and $\det Z_3\neq
0$ we have that
\begin{align*}
\det \beta_{m+4}&=O(1), &\epsilon&\to 0,
\end{align*}
generically.
\end{enumerate}
\end{theorem}
\begin{proof}
  See Appendix \ref{proofs}
\end{proof}

The matrices $Z_1,Z_2$ and $Z_3$ can be expressed in terms of initial conditions as follows
\begin{pro}
  The following expressions in terms of initial conditions hold
  {\small\begin{align*}
  Z_1&=m D_{m,0}^{-1}
-D_{m-1,0}-D_{m,0}-\alpha_{22}-D_{m,0}^{-1}C_{m,0}(B_{m-1,0}+\alpha_{12}),\\
    Z_2&=(m+1)({m}D_{m,0}^{-1}-D_{m,0}^{-1}C_{m,0}(B_{m-1,0}+\alpha_{12})-
D_{m-1,0}-D_{m,0}-\alpha_{22})^{-1}+D_{m,0}^{-1}C_{m,0}B_{m-1,0}
-{m}D_{m,0}^{-1}+D_{m-1,0},\\
    Z_3&=D_{m,0}-(m+1)Z_1^{-1}+(m+2)Z_2^{-1}.
  \end{align*}}
\end{pro}
\begin{proof}
  Is a byproduct of the proof of Theorem \ref{generico}. 
\end{proof}

\appendix
\section{Schur complements}
To show the genericness of the confinement phenomenon in the non Abelian scenario it is very convenient to introduce Schur complements.
\begin{definition}
Given $M$ in block form as in \eqref{Mblocks}, the Schur complements with
respect to $D$ (if $\det D\neq 0$), and to $A$ (if $\det A\neq 0$) are defined
to be
\begin{align*}
  S_{D}(M)&:=A-BD^{-1}C,&
S_{A}(M)&:=D-CA^{-1}B,
\end{align*}
respectively.
\end{definition}
In terms of the Schur complements we have the following well known expressions
for the inverse matrices
{\small\begin{align}
M^{-1}&=
\begin{cases}
\begin{pmatrix}
S_{D}(M)^{-1} &-S_{D}(M)^{-1}BD^{-1}\\[8pt]
-D^{-1}CS_{D}(M)^{-1} & D^{-1}({\mathbb I}_{N-r}+
CS_{D}(M)^{-1}BD^{-1})
\end{pmatrix},
 &\text{for } \det D,  \det S_D(M)\neq 0,\\[20pt]
\begin{pmatrix}
A^{-1}+A^{-1}BS_{A}(M)^{-1}CA^{-1}
 &-A^{-1}BS_{A}(M)^{-1}\\[8pt]
-S_{A}(M)^{-1}CA^{-1} & S_{A}(M)^{-1}
\end{pmatrix},
&\text{for } \det A,  \det S_A(M)\neq 0,\\[20pt]
\begin{pmatrix}
S_{D}(M)^{-1}&-S_{D}(M)^{-1}BD^{-1}\\[20pt]
-D^{-1}CS_{D}(M)^{-1}  & S_{A}(M)^{-1}
\end{pmatrix},
&\text{for }\det A,  \det D,  \det S_D(M), \det S_A(M)\neq 0,
\end{cases}
\label{eq:inversablokes}
\end{align}}
and for the determinant of $M$
\begin{align}
\nn\det M&=\det{A}\hspace{1mm}  {\det}S_{A}(M)\\
&=\det{D} \hspace{1mm}  {\det}S_{D}(M).
\label{eq:detblokesD}
\end{align}
Now, if $K=\Big(\begin{smallmatrix}
A&B\\C&D
\end{smallmatrix}\Big)=\Big(\begin{smallmatrix}
0&0\\C_0&D_0
\end{smallmatrix}\Big)+\Big(\begin{smallmatrix}
A_1&B_1\\C_1&D_1
\end{smallmatrix}\Big)\epsilon+O(\epsilon^2)\in\mathcal A_{\mathfrak K}$ then we can write the Schur complements in the form
\begin{align}
 S_{D}(K)&=A-BD^{-1}C=:S_{D}(K)_1\epsilon+S_{D}(K)_2\epsilon^{2}
 +O(\epsilon^3),& \epsilon&\to0,\label{eq:schur1}\\
\nn{}S_{A}(K)&=D-CA^{-1}B=:S_{A}(K)_0+S_{A}(K)_1\epsilon+O(\epsilon^2),&
\epsilon&\to0,
\end{align}
 where
\begin{align*}
 S_{D}(K)_1=&A_{1}-B_{1}D_{0}^{-1}C_{0},\\
 S_{D}(K)_2=&A_{2}-B_{1}D_{0}^{-1}C_{1}-  B_{2}D_{0}^{-1}C_{0}+B_{1}D_{0}^{-1}D_{1}D_{0}^{-1}C_{0},\\
 S_{D}(K)_3=&A_{3}-B_{1}D_{0}^{-1}C_{2}+(B_{1}D_{0}^{-1}D_{1}D_{0}^{-1}- B_{2}D_{0}^{-1})C_{1} +
B_{1}(D_{0}^{-1}D_{2}D_{0}^{-1}-D_{0}^{-1}D_{1}D_{0}^{-1}
D_{1}D_{0}^{-1})C_{0} \\ \nn &+ B_2D_{0}^{-1}D_{1}D_{0}^{-1}C_{0}
-B_{3}D_{0}^{-1}C_{0},\\
S_{D}(K)_4=&A_{4}-B_{1}D_{0}^{-1}C_{3}+B_{1}D_{0}^{-1}D_{1}D_{0}^{-1}C_{2}
-B_{1}D_{0}^{-1}(D_{1}D_{0}^{-1}D_{1}D_{0}^{-1}-
D_{2}D_{0}^{-1})C_{1}-B_{1}D_{0}^{-1}D_{2}D_{0}^{-1}D_{1}D_{0}^{-1}C_{0}
\\ \nn &+B_{1}D_{0}^{-1}D_{1}(D_{0}^{-1}D_{1}D_{0}^{-1}D_{1}D_{0}^{-1}-
D_{0}^{-1}D_{2}D_{0}^{-1})C_{0}+B_{1}D_{0}^{-1}D_{3}D_{0}^{-1}C_{0}
-B_{2}D_{0}^{-1}C_{2}\\ \nn &+B_{2}D_{0}^{-1}D_{1}D_{0}^{-1}C_{1}-
B_{2}D_{0}^{-1}(D_{1}D_{0}^{-1}D_{1}D_{0}^{-1}-D_{2}D_{0}^{-1})C_{0}
-B_{3}D_{0}^{-1}(C_{1}-D_{1}D_{0}^{-1}C_{0})-B_{4}D_{0}^{-1}C_{0},\\
S_{A}(K)_0=&D_{0}-C_{0}A_{1}^{-1}B_{1},\nonumber\\
 S_{A}(K)_1=&D_{1}-C_{0}A_{1}^{-1}B_{2}-C_{1}
A_{1}^{-1}B_{1}+C_{0}A_{1}^{-1}A_{2}A_{1}^{-1}B_{1}.% \label{eq:SD}
\end{align*}
For the determinant $\det M$ we just take into account eqs. \eqref{eq:detblokesD} and \eqref{eq:schur1} to get
\begin{align*}
\det K=&\epsilon^{r}\det(A_{1}-B_{1}D_{0}^{-1}C_{0}
+O(\epsilon))\det(D_{0}+O(\epsilon))\\
=&\det(A_{1}-B_{1}D_{0}^{-1}C_{0})\det(D_{0})\epsilon^{r}+O(\epsilon^{r+1}).
\end{align*}

\section{Proof of Theorem \ref{generico}}\label{proofs}
\begin{lemma}\label{pro:beta1}
 \begin{enumerate}
\item Assuming that  $\det D_{m,0}\neq 0$  the following asymptotic holds
\begin{align*}
\det\beta_{m+1}=&\epsilon^{-r}\begin{vmatrix}
  mS_D(\beta_m)_1^{-1}&
-mS_D(\beta_m)^{-1}_1B_{m,1}D_{m,0}^{-1}-B_{m-1,0}-\alpha_{12}
  \\
  -mD_{m,0}^{-1}C_{m,0}S_D(\beta_m)_1^{-1}&{m}D_{m,0}^{-1}
+{m}D_{m,0}^{-1}C_{m,0}S_{D}(\beta_m)_{1}^{-1}B_{m,1}D_{m,0}^{-1}
-D_{m-1,0}-D_{m,0}-\alpha_{22}
\end{vmatrix}\\&+O(\epsilon^{-r+1})
\end{align*}
for $\epsilon\to0$, where $S_{D}(\beta_m)_1:=A_{m,1}-B_{m,1}D_{m,0}^{-1}C_{m,0}\in\C^{r\times r}$.
\end{enumerate}
 \end{lemma}
 \begin{proof}
From eq. \eqref{eq:betauno} we know that
\begin{align*}
\det\begin{pmatrix}
A_{m,1} & B_{m,1}\\
C_{m,0} & D_{m,0}
\end{pmatrix}\neq0,
\end{align*}
hence $S_D(\beta_m)_1$ is invertible.
 Then, from \eqref{eq:inversablokes}
and \eqref{eq:schur1} we deduce
\begin{align*}
\beta_{m}^{-1}=\begin{pmatrix}
(\beta_{m}^{-1})_{11,-1}&0\\[8pt]
(\beta_{m}^{-1})_{21,-1}&0
\end{pmatrix}\epsilon^{-1}+
\begin{pmatrix}
(\beta_{m}^{-1})_{11,0}& (\beta_{m}^{-1})_{12,0}\\[8pt]
(\beta_{m}^{-1})_{21,0}& (\beta_{m}^{-1})_{22,0}
\end{pmatrix}+
\begin{pmatrix}
(\beta_{m}^{-1})_{11,1}& (\beta_{m}^{-1})_{12,1}\\[8pt]
(\beta_{m}^{-1})_{21,1}& (\beta_{m}^{-1})_{22,1}
\end{pmatrix}\epsilon+O(\epsilon^2)
, &&&\epsilon\to 0,
\end{align*}
where the pole  coefficients are
 \begin{align}\label{betainv-1}
 (\beta_{m}^{-1})_{11,-1}:=&S_{D}(\beta_m)_{1}^{-1},&
 (\beta_{m}^{-1})_{21,-1}:=&-D_{m,0}^{-1}C_{m,0}S_{D}(\beta_m)_{1}^{-1},
 \end{align}
 while the regular part coefficients are
\begin{align*}
(\beta_{m}^{-1})_{11,0}:=&-S_{D}(\beta_m)_{1}^{-1}
S_{D}(\beta_m)_{2}S_{D}(\beta_m)_{1}^{-1},\\ (\beta_{m}^{-1})_{12,0}:=&
-S_{D}(\beta_m)_{1}^{-1}B_{m,1}D_{m,0}^{-1},\\
(\beta_{m}^{-1})_{21,0}:=&D_{m,0}^{-1}(C_{m,0}S_{D}(\beta_m)_{1}^{-1}
S_{D}(\beta_m)_{2}
S_{D}(\beta_m)_{1}^{-1}- (C_{m,1}
-D_{m,1}D_{m,0}^{-1}C_{m,0})S_{D}(\beta_m)_{1}^{-1}),\\
 (\beta_{m}^{-1})_{22,0}:=&D_{m,0}^{-1}\big(\mathbb{I}_{N-r}
+C_{m,0}S_{D}(\beta_m)_{1}^{-1}B_{m,1}D_{m,0}^{-1}\big),
\end{align*}
\begin{align*}%\label{betainv1}
(\beta_{m}^{-1})_{11,1}:=&S_{D}(\beta_m)_{1}^{-1}
S_{D}(\beta_m)_{2}S_{D}(\beta_m)_{1}^{-1}S_{D}(\beta_m)_{2}
S_{D}(\beta_m)_{1}^{-1}-S_{D}(\beta_m)_{1}^{-1}S_{D}(\beta_m)_{3}
S_{D}(\beta_m)_{1}^{-1},\\
(\beta_{m}^{-1})_{12,1}:=&\big(S_{D}(\beta_m)_{1}^{-1}
S_{D}(\beta_m)_{2}S_{D}(\beta_m)_{1}^{-1}B_{m,1}
 -S_{D}(\beta_m)_{1}^{-1}(B_{m,2}-
B_{m,1}D_{m,0}^{-1}D_{m,1})\big)D_{m,0}^{-1},\\
(\beta_{m}^{-1})_{21,1}:=&-D_{m,0}^{-1}\Big(C_{m,0}[S_{D}(\beta_m)_{1}^{-1}
  S_{D}(\beta_m)_{2}S_{D}(\beta_m)_{1}^{-1}S_{D}(\beta_m)_{2}
  S_{D}(\beta_m)_{1}^{-1}-S_{D}(\beta_m)_{1}^{-1}
  S_{D}(\beta_m)_{3}S_{D}(\beta_m)_{1}^{-1}]\\  &-(C_{m,1}
- D_{m,1}D_{m,0}^{-1}C_{m,0})S_{D}(\beta_m)_{1}^{-1}
S_{D}(\beta_m)_{2}S_{D}(\beta_m)_{1}^{-1}+
\\  &- \big((D_{m,1}D_{m,0}^{-1}D_{m,1}-
 D_{m,2})D_{m,0}^{-1}C_{m,0} +C_{m,2}
  -D_{m,1}D_{m,0}^{-1}C_{m,1}\big)S_{D}(\beta_m)_{1}^{-1}\Big),\\
 (\beta_{m}^{-1})_{22,1}:=&D_{m,0}^{-1}\Big(-D_{m,1}+(C_{m,1}
-D_{m,1}D_{m,0}^{-1}C_{m,0}- C_{m,0}S_{D}(\beta_m)_{1}^{-1}S_{D}(\beta_m)_{2})
S_{D}(\beta_m)_{1}^{-1}B_{m,1}
\\ &+C_{m,0}S_{D}(\beta_m)_{1}^{-1}(B_{m,2}-
B_{m,1}D_{m,0}^{-1}D_{m,1})\Big)D_{m,0}^{-1},\\
(\beta_{m}^{-1})_{11,2}:=&S_{D}(\beta_m)_{1}^{-1}S_{D}(\beta_m)_{3}
S_{D}(\beta_m)_{1}^{-1}S_{D}(\beta_m)_{2}S_{D}(\beta_m)_{1}^{-1}
\\  &-S_{D}(\beta_m)_{1}^{-1}S_{D}(\beta_m)_{2}S_{D}(\beta_m)_{1}^{-1}(
S_{D}(\beta_m)_{2}S_{D}(\beta_m)_{1}^{-1}S_{D}(\beta_m)_{2}S_{D}(\beta_m)_{1}^{-1}
\\  &-S_{D}(\beta_m)_{3}S_{D}(\beta_m)_{1}^{-1})-
S_{D}(\beta_m)_{1}^{-1}S_{D}(\beta_m)_{4}S_{D}(\beta_m)_{1}^{-1},\\
(\beta_{m}^{-1})_{12,2}:=&S_{D}(\beta_m)_{1}^{-1}B_{m,1}
D_{m,0}^{-1}(D_{m,2}D_{m,0}^{-1}-D_{m,1}D_{m,0}^{-1}D_{m,1}D_{m,0}^{-1})
+S_{D}(\beta_m)_{1}^{-1}(B_{m,2}\\  &-S_{D}(\beta_m)_{2}
S_{D}(\beta_m)_{1}^{-1}B_{m,1})D_{m,0}^{-1}D_{m,1}D_{m,0}^{-1}-
S_{D}(\beta_m)_{1}^{-1}(B_{m,3}-
S_{D}(\beta_m)_{2}S_{D}(\beta_m)_{1}^{-1}B_{m,2}\\  &+
S_{D}(\beta_m)_{2}S_{D}(\beta_m)_{1}^{-1}S_{D}(\beta_m)_{2}
S_{D}(\beta_m)_{1}^{-1}B_{m,1}-S_{D}(\beta_m)_{3}S_{D}(\beta_m)_{1}^{-1}
B_{m,1})D_{m,0}^{-1}.
\end{align*}

 Finally, from  eq. \eqref{eq:painleveS} we deduce
\begin{align}
\beta_{m+1}=\begin{pmatrix}
{m}S_{D}(\beta_m)_{1}^{-1}& 0\\
-{m}D_{m,0}^{-1}C_{m,0}S_{D}(\beta_m)_{1}^{-1}
& 0
\end{pmatrix}\epsilon^{-1}+
\begin{pmatrix}
   A_{m+1,0}  & B_{m+1,0} \\
   C_{m+1,0}    & D_{m+1,0}
\end{pmatrix}+
\begin{pmatrix}
A_{m+1,1} & B_{m+1,1}\\
C_{m+1,1}& D_{m+1,1}
\end{pmatrix}\epsilon+O(\epsilon^2),\epsilon\to 0,
\label{eq:bloques1}\end{align}
 where, in terms of eqs. \eqref{betainv-1}--\eqref{betainv1},
\begin{align}
  A_{m+1,0}:=&{m}(\beta_{m}^{-1})_{11,0}-A_{m-1,0}-\alpha_{11},
&
B_{m+1,0}:=&{m}(\beta_{m}^{-1})_{12,0}-B_{m-1,0}
-\alpha_{12},\label{eq:bloques1b0}\\
 C_{m+1,0}:=&{m}(\beta_{m}^{-1})_{21,0}-
C_{m-1,0}-C_{m,0}-\alpha_{21},
&
 D_{m+1,0}:=&{m}(\beta_{m}^{-1})_{22,0}
-D_{m-1,0}-D_{m,0}-\alpha_{22},
\label{eq:bloques1d0}
\\
A_{m+1,1}:=&{m}(\beta_{m}^{-1})_{11,1}
-A_{m-1,1}-A_{m,1},&
\label{eq:bloques1b1}B_{m+1,1}:=&{m}(\beta_{m}^{-1})_{12,1}
-B_{m-1,1}-B_{m,1},
\\
C_{m+1,1}:=&{m}(\beta_{m}^{-1})_{21,1}-C_{m-1,1}-C_{m,1},&%\label{eq:bloques1c1}\\
D_{m+1,1}:=&{m}(\beta_{m}^{-1})_{22,1}-D_{m-1,1}-D_{m,1},\label{eq:bloques1d1}
\\
A_{m+1,2}:=&{m}(\beta_{m}^{-1})_{11,2}
-A_{m-1,2}-A_{m,2},&
B_{m+1,2}:=&{m}(\beta_{m}^{-1})_{12,2}-B_{m-1,2}-B_{m,2}.
\label{eq:bloques1a2}
\end{align}
Observing that
\begin{align*}
\det\beta_{m+1}&=\begin{vmatrix}
{m}S_{D}(\beta_m)_{1}^{-1}& B_{m+1,0}\\
-{m}D_{m,0}^{-1}C_{m,0}S_{D}(\beta_m)_{1}^{-1}
& D_{m+1,0}
\end{vmatrix}\epsilon^{-r}+
O(\epsilon^{-r+1})
,&\epsilon&\to 0,
\end{align*}
the result follows.
\end{proof}

Now observe that
\begin{align*}
  Z_1:=&
{m}D_{m,0}^{-1}+{m}D_{m,0}^{-1}C_{m,0}S_{D}(\beta_m)_{1}^{-1}B_{m,1}D_{m,0}^{-1}-
D_{m-1,0}-D_{m,0}-\alpha_{22}
\\&-(  -mD_{m,0}^{-1}C_{m,0}S_D(\beta_m)_1^{-1})(  mS_D(\beta_m)_1^{-1})^{-1}(
-mS_D(\beta_m)^{-1}_1B_{m,1}D_{m,0}^{-1}-B_{m-1,0}-\alpha_{12})\\=&m D_{m,0}^{-1}
-D_{m-1,0}-D_{m,0}-\alpha_{22}-D_{m,0}^{-1}C_{m,0}(B_{m-1,0}+\alpha_{12}).
\end{align*}

By using the determinant expansion in Schur complements of Lemma \ref{pro:beta1}, one observes that
  \begin{align*}
    \begin{vmatrix}
  mS_D(\beta_m)_1^{-1}&-mS_D(\beta_m)^{-1}_1B_{m,1}D_{m,0}^{-1}
  -B_{m-1,0}-\alpha_{12}
\\
  -mD_{m,0}^{-1}C_{m,0}S_D(\beta_m)_1^{-1}&{m}D_{m,0}^{-1}+
  {m}D_{m,0}^{-1}C_{m,0}S_{D}(\beta_m)_{1}^{-1}B_{m,1}D_{m,0}^{-1}
  -D_{m-1,0}-D_{m,0}-\alpha_{22}
\end{vmatrix}\\=\det\Big(mS_D(\beta_m)_1^{-1}\Big)\det Z_1.
\end{align*}
and the first point of the Theorem is proved.

Let us now go one step further in the discrete matrix chain and move to
position $m+2$.
\begin{lemma}\label{pro:beta2}
Whenever $\det D_{m,0}\neq 0$ and $\det Z_1\neq 0$ the following asymptotic
hold
   \begin{align*}
\det\beta_{m+2}=&\epsilon^{-r}\begin{vmatrix}
  -mS_D(\beta_m)_1^{-1}&
mS_D(\beta_m)^{-1}_1B_{m,1}D_{m,0}^{-1}+B_{m-1,0}
  \\
  {m}D_{m,0}^{-1}C_{m,0}S_D(\beta_m)_1^{-1}&(m+1)Z_{1}^{-1}-{m}D_{m,0}^{-1}
-{m}D_{m,0}^{-1}C_{m,0}S_{D}(\beta_m)_{1}^{-1}B_{m,1}D_{m,0}^{-1}
+D_{m-1,0}
\end{vmatrix}\\&+O(\epsilon^{-r+1})
\end{align*}
for $\epsilon\to0$.
 \end{lemma}
\begin{proof}
As $\det\beta_{m+1}=O(\epsilon^{-r})$, $\epsilon\rightarrow0$, and consequently point (2) of Proposition
\ref{pro:rings} tells us that $\beta_{m+1}^{-1}\in\mathcal A_{\mathfrak K}$.  Therefore, the following asymptotic expansion for the inverse matrix holds
\begin{align}
\beta_{m+1}^{-1}=\begin{pmatrix}
0 & 0\\
(\beta_{m+1}^{-1})_{21,0} & (\beta_{m+1}^{-1})_{22,0}
\end{pmatrix}+
\begin{pmatrix}
(\beta_{m+1}^{-1})_{11,1}   & (\beta_{m+1}^{-1})_{12,1}  \\
(\beta_{m+1}^{-1})_{21,1}  & (\beta_{m+1}^{-1})_{22,1}
\end{pmatrix}\epsilon+
\begin{pmatrix}
(\beta_{m+1}^{-1})_{11,2}  & (\beta_{m+1}^{-1})_{12,2}\\
(\beta_{m+1}^{-1})_{21,2} & (\beta_{m+1}^{-1})_{22,2}
\end{pmatrix}\epsilon^2+O(\epsilon^3),
\label{eq:invbmas1}
\end{align}
for $\epsilon\to 0$. Here the blocks   $(\beta_{m+1}^{-1})_{ab,j}$
 are to be found from
%\begin{align*}
%\beta_{m+1}\beta_{m+1}^{-1}={\mathbb I_N},
%\end{align*}
the asymptotic expansion \eqref{eq:bloques1}. We conclude
 \begin{align*}
(\beta_{m+1}^{-1})_{21,0}=&Z_{1}^{-1}D_{m,0}^{-1}C_{m,0},&(\beta_{m+1}^{-1})_{22,0}=&Z_{1}^{-1}, \\
(\beta_{m+1}^{-1})_{11,1}=&\frac{1}{m}S_{D}(\beta_{m})_{1}-
\frac{1}{m}S_{D}(\beta_{m})_{1}B_{m+1,0}Z_{1}^{-1}D_{m,0}^{-1}C_{m,0},&
 (\beta_{m+1}^{-1})_{12,1}=&
-\frac{1}{m}S_{D}(\beta_{m})_{1}B_{m+1,0}Z_{1}^{-1},\end{align*}
\begin{align*}
(\beta_{m+1}^{-1})_{21,1}=&-Z_{1}^{-1}D_{m,0}^{-1}C_{m,0}
-\frac{1}{m}Z_{1}^{-1}(C_{m+1,0}+D_{m,0}^{-1}C_{m,0}A_{m+1,0})
S_{D}(\beta_{m})_{1}(\mathbb{I}_r-
B_{m+1,0}Z_{1}^{-1}D_{m,0}^{-1}C_{m,0}),\\
(\beta_{m+1}^{-1})_{22,1}=&-Z_{1}^{-1}
+\frac{1}{m}Z_{1}^{-1}(C_{m+1,0}+D_{m,0}^{-1}C_{m,0}A_{m+1,0})
S_{D}(\beta_{m})_{1}B_{m+1,0}Z_{1}^{-1},\\
(\beta_{m+1}^{-1})_{11,2}=&
-\frac{1}{m^2}S_{D}(\beta_{m})_{1}A_{m+1,0}S_{D}(\beta_{m})_{1}+
\frac{1}{m^2}S_{D}(\beta_{m})_{1}A_{m+1,0}S_{D}(\beta_{m})_{1}B_{m+1,0}
Z_{1}^{-1}D_{m,0}^{-1}C_{m,0}\\&+\frac{1}{m^2}
S_{D}(\beta_{m})_{1}B_{m+1,0}Z_{1}^{-1}(C_{m+1,0}+D_{m,0}^{-1}C_{m,0}A_{m+1,0})
S_{D}(\beta_{m})_{1}(\mathbb{I}_r-
B_{m+1,0}Z_{1}^{-1}D_{m,0}^{-1}C_{m,0}),\\
(\beta_{m+1}^{-1})_{12,2}=&
-\frac{1}{m^2}S_{D}(\beta_{m})_{1}B_{m+1,0}Z_{1}^{-1}(C_{m+1,0}+
D_{m,0}^{-1}C_{m,0}A_{m+1,0})S_{D}(\beta_{m})_{1}B_{m+1,0}Z_{1}^{-1}
\\&+\frac{1}{m^2}S_{D}(\beta_{m})_{1}A_{m+1,0}S_{D}(\beta_{m})_{1}B_{m+1,0}Z_{1}^{-1}.
\end{align*}
 If we substitute equations
(\ref{eq:bloques1b0})-(\ref{eq:bloques1a2}) into eq. (\ref{eq:painleveS}),
we have that for $\epsilon\to 0$
\begin{align}
\beta_{m+2}=\begin{pmatrix}
{-m}S_{D}(\beta_m)_{1}^{-1} & 0\\
{m}D_{m,0}^{-1}C_{m,0}S_{D}(\beta_m)_{1}^{-1}& 0
\end{pmatrix}\epsilon^{-1}+
\begin{pmatrix}
A_{m+2,0} & B_{m+2,0}\\
C_{m+2,0}  & D_{m+2,0}
\end{pmatrix}+
\begin{pmatrix}
A_{m+2,1} & B_{m+2,1}\\
C_{m+2,1}  & D_{m+2,1}
\end{pmatrix}\epsilon+O(\epsilon^2),
\label{eq:bloques2}
\end{align}
 where
\begin{align*}
A_{m+2,0}:=&-A_{m+1,0}-\alpha_{11},&
B_{m+2,0}:=&-B_{m+1,0}-\alpha_{12},\\
C_{m+2,0}:=&(m+1)(\beta_{m+1}^{-1})_{21,0}-C_{m+1,0}-C_{m,0}-\alpha_{21},&
D_{m+2,0}:=&(m+1)(\beta_{m+1}^{-1})_{22,0}-D_{m+1,0}-D_{m,0}-\alpha_{22},\\
A_{m+2,1}:=&(m+1)(\beta_{m+1}^{-1})_{11,1}-A_{m+1,1}-A_{m,1},&
B_{m+2,1}:=&(m+1)(\beta_{m+1}^{-1})_{12,1}-B_{m+1,1}-B_{m,1},\\
C_{m+2,1}:=&(m+1)(\beta_{m+1}^{-1})_{21,1}-C_{m+1,1}-C_{m,1},&
D_{m+2,1}:=&(m+1)(\beta_{m+1}^{-1})_{22,1}-D_{m+1,1}-D_{m,1},\\
A_{m+2,2}:=&(m+1)(\beta_{m+1}^{-1})_{11,2}-A_{m+1,2}-A_{m,2},&
B_{m+2,2}:=&(m+1)(\beta_{m+1}^{-1})_{12,2}-B_{m+1,2}-B_{m,2}.
\end{align*}
Now, observing that
\begin{align*}
\det\beta_{m+2}=&\begin{vmatrix}
-{m}S_{D}(\beta_m)_{1}^{-1}& B_{m+2,0}\\
{m}D_{m,0}^{-1}C_{m,0}S_{D}(\beta_m)_{1}^{-1}
& D_{m+2,0}
\end{vmatrix}\epsilon^{-r}+
O(\epsilon^{-r+1})
,&\epsilon&\to 0,
\end{align*}
the result follows.
\end{proof}
Notice that
\begin{multline*}
  Z_2:=
(m+1)({m}D_{m,0}^{-1}-D_{m,0}^{-1}C_{m,0}(B_{m-1,0}+\alpha_{12})-
D_{m-1,0}-D_{m,0}-\alpha_{22})^{-1}\\+D_{m,0}^{-1}C_{m,0}B_{m-1,0}
-{m}D_{m,0}^{-1}+D_{m-1,0}.
\end{multline*}

We expand the determinant according to
  Schur complements, obtaining
  \begin{align*}\\
\begin{vmatrix}
  -mS_D(\beta_m)_1^{-1}&
mS_D(\beta_m)^{-1}_1B_{m,1}D_{m,0}^{-1}+B_{m-1,0}
  \\
  mD_{m,0}^{-1}C_{m,0}S_D(\beta_m)_1^{-1}&(m+1)Z_{1}^{-1}-{m}D_{m,0}^{-1}
-{m}D_{m,0}^{-1}C_{m,0}S_{D}(\beta_m)_{1}^{-1}B_{m,1}D_{m,0}^{-1}
+D_{m-1,0}
\end{vmatrix}\\=\det\Big(-mS_D(\beta_m)_1^{-1}\Big)\det Z_2
\end{align*}
from which the second point of the Theorem follows immediately.

\begin{lemma}\label{pro:beta33}
Assuming that $\det D_{m,0}\neq 0$, $\det Z_1\neq 0$ and $\det Z_2\neq 0$ the
following asymptotic expansion for $\epsilon \to 0$ holds
  {\small\begin{align*}
\det\beta_{m+3}=&\epsilon^{r}\begin{vmatrix}
  (m+2)(\beta_{m+2}^{-1})_{11,1}
-(m+1)(\beta_{m+1}^{-1})_{11,1}+A_{m,1} & (m+2)(\beta_{m+2}^{-1})_{12,1}
-(m+1)(\beta_{m+1}^{-1})_{12,1}+B_{m,1}
  \\
  (m+2)(\beta_{m+2}^{-1})_{21,0}
-(m+1)(\beta_{m+1}^{-1})_{21,0}+C_{m,0}& (m+2)(\beta_{m+2}^{-1})_{22,0}
-(m+1)(\beta_{m+1}^{-1})_{22,0}+D_{m,0}
\end{vmatrix}\\&+O(\epsilon^{r+1}),
\end{align*}}
where
\begin{align*}
(\beta_{m+2}^{-1})_{21,0}:=&Z_{2}^{-1}D_{m,0}^{-1}C_{m,0},
&
(\beta_{m+2}^{-1})_{22,0}:=&Z_{2}^{-1},\\
(\beta_{m+2}^{-1})_{11,1}:=&-\frac{1}{m}S_{D}(\beta_m)_{1}(\mathbb{I}_r-
B_{m+2,0}Z_{2}^{-1}D_{m,0}^{-1}C_{m,0}),&
(\beta_{m+2}^{-1})_{12,1}:=&\frac{1}{m}S_{D}(\beta_m)_{1}B_{m+2,0}Z_{2}^{-1},
\end{align*}
\begin{align*}
(\beta_{m+2}^{-1})_{21,1}:=&
-Z_{2}^{-1}D_{m,0}^{-1}C_{m,0}+\frac{1}{m}Z_{2}^{-1}(C_{m+2,0}+
D_{m,0}^{-1}C_{m,0}A_{m+2,0})S_{D}(\beta_{m})_{1}(\mathbb{I}_r
-B_{m+2,0}Z_{2}^{-1}D_{m,0}^{-1}C_{m,0}),\\
(\beta_{m+2}^{-1})_{22,1}:=&-Z_{2}^{-1}-\frac{1}{m}Z_{2}^{-1}(C_{m+2,0}+
D_{m,0}^{-1}C_{m,0}A_{m+2,0})S_{D}(\beta_{m})_{1}B_{m+2,0}Z_{2}^{-1},\\
(\beta_{m+2}^{-1})_{11,2}:=&
\frac{1}{m^2}S_{D}(\beta_{m})_{1}B_{m+2,0}Z_{2}^{-1}(C_{m+2,0}+
D_{m,0}^{-1}C_{m,0}A_{m+2,0})S_{D}(\beta_{m})_{1}(\mathbb{I}_r-
B_{m+2,0}Z_{2}^{-1}D_{m,0}^{-1}C_{m,0}) \\&-
\frac{1}{m^2}S_{D}(\beta_{m})_{1}A_{m+2,0}S_{D}(\beta_{m})_{1}(\mathbb{I}_r-
B_{m+2,0}Z_{2}^{-1}D_{m,0}^{-1}C_{m,0}) ,\\
(\beta_{m+2}^{-1})_{12,2}:=&
\frac{1}{m^2}S_{D}(\beta_{m})_{1}A_{m+2,0}S_{D}(\beta_{m})_{1}
B_{m+2,0}Z_{2}^{-1}\\&-
\frac{1}{m^2}S_{D}(\beta_{m})_{1}B_{m+2,0}Z_{2}^{-1}(C_{m+2,0}+
D_{m,0}^{-1}C_{m,0}A_{m+2,0})S_{D}(\beta_{m})_{1}B_{m+2,0}Z_{2}^{-1}.
\end{align*}
\end{lemma}
\begin{proof}
From equation \eqref{eq:bloques2} we get that $\beta_{m+2}\in{\mathbb L}$. Therefore, since ${\det}$ $Z_{2}{\neq}0$,
we have
\begin{align}
\beta_{m+2}^{-1}=\begin{pmatrix}
0 & 0\\
(\beta_{m+2}^{-1})_{21,0} & (\beta_{m+2}^{-1})_{22,0}
\end{pmatrix}+
\begin{pmatrix}
(\beta_{m+2}^{-1})_{11,1}   & (\beta_{m+2}^{-1})_{12,1}  \\
(\beta_{m+2}^{-1})_{21,1}  & (\beta_{m+2}^{-1})_{22,1}
\end{pmatrix}\epsilon+
\begin{pmatrix}
(\beta_{m+2}^{-1})_{11,2}  & (\beta_{m+2}^{-1})_{12,2}\\
(\beta_{m+2}^{-1})_{21,2} & (\beta_{m+2}^{-1})_{22,2}
\end{pmatrix}\epsilon^2+O(\epsilon^3),
\label{eq:invbmas2}
\end{align}
 where the blocks $(\beta_{m+2}^{-1})_{ab,j}$, are determined by
 the
 asymptotic expansion \eqref{eq:bloques2}.  If we substitute
 \eqref{eq:bloques1},
 \eqref{eq:bloques2} and (\ref{eq:invbmas2}) into the
 matrix equation \eqref{eq:painleveS},
 we have that
\begin{align*}
\beta_{m+3}=\begin{pmatrix}
0 & 0\\
C_{m+3,0} & D_{m+3,0}
\end{pmatrix}+
\begin{pmatrix}
A_{m+3,1}   & B_{m+3,1}  \\
C_{m+3,1}  & D_{m+3,1}
\end{pmatrix}\epsilon+
\begin{pmatrix}
A_{m+3,2}  & B_{m+3,2}\\
C_{m+3,2} & D_{m+3,2}
\end{pmatrix}\epsilon^2+O(\epsilon^3),
%\label{eq:bloques3}
\end{align*}
 where
 {\small\begin{align*}
   C_{m+3,0}:=&(m+2)(\beta_{m+2}^{-1})_{21,0}-(m+1)(\beta_{m+1}^{-1})_{21,0}+C_{m,0},
   &
   D_{m+3,0}:=&(m+2)(\beta_{m+2}^{-1})_{22,0}-(m+1)(\beta_{m+1}^{-1})_{22,0}+D_{m,0},\\
   A_{m+3,1}:=&(m+2)(\beta_{m+2}^{-1})_{11,1}-(m+1)(\beta_{m+1}^{-1})_{11,1}+A_{m,1},&
   B_{m+3,1}:=&(m+2)(\beta_{m+2}^{-1})_{12,1}-(m+1)(\beta_{m+1}^{-1})_{12,1}+B_{m,1},&
   \\
   C_{m+3,1}:=&(m+2)(\beta_{m+2}^{-1})_{21,1}-(m+1)(\beta_{m+1}^{-1})_{21,1}+C_{m,1},&
   D_{m+3,1}:=&(m+2)(\beta_{m+2}^{-1})_{22,1}-(m+1)(\beta_{m+1}^{-1})_{22,1}+D_{m,1},\\
   A_{m+3,2}:=&(m+2)(\beta_{m+2}^{-1})_{11,2}-(m+1)(\beta_{m+1}^{-1})_{11,2}+A_{m,2},
   &
   B_{m+3,2}:=&(m+2)(\beta_{m+2}^{-1})_{12,2}-(m+1)(\beta_{m+1}^{-1})_{12,2}+B_{m,2}.
\end{align*}}
Then, if we use again Proposition \ref{pro:rings}, we  deduce
{\small \begin{align}
\det \beta_{m+3}&=\epsilon^{r}\begin{vmatrix}
A_{m+3,1} & B_{m+3,1}\\
C_{m+3,0} & D_{m+3,0}
\end{vmatrix}+O(\epsilon^{r+1}), & \epsilon&\to 0,
\label{eq:detbetam3}
\end{align}}
and the result follows.
\end{proof}
\noi Notice that
\begin{align*}
  Z_3=D_{m,0}-(m+1)Z_1^{-1}+(m+2)Z_2^{-1}.
\end{align*}
Notice the similarity with eq. \eqref{eq:painleves}.

Taking into account that
\begin{align}\label{eq:m+3}
  C_{m+3,0}=&Z_3D_{m,0}^{-1} C_{m,0}, & D_{m+3,0}=Z_3,
\end{align}
we express the determinant in equation \eqref{eq:detbetam3}  as follows
\begin{align}
\begin{vmatrix}
A_{m+3,1} & B_{m+3,1}\\
C_{m+3,0} & D_{m+3,0}
\end{vmatrix}={\det}Z_3\det(A_{m+3,1}-B_{m+3,1}D_{m,0}^{-1}C_{m,0}),
\end{align}
where \begin{align*}
  A_{m+3,1}-B_{m+3,1}D_{m,0}^{-1}C_{m,0}=
-\frac{(m+3)}{m}S_{D}(\beta_m)_1.
\end{align*}
 This implies that the determinant in equation \eqref{eq:detbetam3} vanishes
 if and only if
\begin{align*}
\det Z_3=0.
\end{align*}

Finally, under the previous hypotheses, eqs. \eqref{eq:invert1}-\eqref{eq:invert3}
hold. As a by product of the proof of Theorem \ref{main}, we get that
\begin{align*}
\beta_{m+4}=\beta_{m+3}^{-1} A-\beta_{m+3}-\alpha,
\end{align*}
where $\beta_{m+3}, A \in\mathcal A_{\mathfrak K}$ and
$(\beta_{m+3})^{-1}\in\epsilon^{-1}\mathcal A_{\mathfrak L}$.  According to
Proposition \ref{pro:rings} (6), $\beta_{m+3}^{-1} A\in \mathcal A$, so that
we can write
\begin{align*}
\beta_{m+4}&=O(1), & \epsilon&\to 0.
\end{align*}
We can write  the matrix dynamical system \eqref{eq:painleveS} as
\begin{align}
\beta_{n-1}=n\beta_{n}^{-1}-\beta_{n+1}-\beta_{n}-\alpha,
\label{eq:painleveS2}
\end{align}
which can be seen as the application of a \textit{time reversal symmetry}. From $\beta_{m+4}\in\mathcal A$ and $\beta_{m+3}\in\mathcal A_{\mathfrak K}$, understood now as initial
conditions, we get the quantities $\beta_{m+2}$, $\beta_{m+1}$, $\beta_{m}$ and $\beta_{m-1}$. Observe that our initial assumption was precisely that
$\beta_{m-1}\in\mathcal A$ and $\beta_m\in\mathcal A_{\mathfrak K}$, see \eqref{eq:betacero} and \eqref{eq:betauno}. Hence, the whole forward process,
and its conclusions about the asymptotic behaviours, can be reversed backwards. Consequently, since the assumption that $\det\beta_{m+4,0}=0$  reduces the number of free parameters from $N^2$ to $N^2-1$, we conclude that $\beta_{m-1,0}$ involves at most $N^2-1$ free parameters (if no
further constraint is requested). This is in contradiction to our departing hypothesis that $\beta_{m-1,0}$ has $N^2$ free parameters.  Therefore $\det
\beta_{m+4}=O(1)$ as $\epsilon\to 0$ generically.

\end{document}